\documentclass[11pt]{amsart}

\usepackage[T1]{fontenc}
\usepackage{lmodern}
\usepackage{microtype}
\usepackage{amsmath,amssymb,mathtools}
\usepackage{mathrsfs}
\usepackage{enumitem}
\usepackage[hidelinks]{hyperref}
\hypersetup{
  pdftitle={Bounded Representatives in Critical Sobolev--Hodge Spaces},
  pdfauthor={Xinan Dai, Wenhao Deng, Yingdong Shi, Tailin Wu, Yuchen Yang},
  pdfsubject={Critical Sobolev--Hodge selection and endpoint estimates},
  pdfkeywords={critical Sobolev spaces, differential forms, Hodge decomposition, Bourgain--Brezis estimates, Riesz potentials, Mazya Phi inequalities}
}

\allowdisplaybreaks
\numberwithin{equation}{section}

\newtheorem{theorem}{Theorem}[section]
\newtheorem{proposition}[theorem]{Proposition}
\newtheorem{lemma}[theorem]{Lemma}
\newtheorem{corollary}[theorem]{Corollary}

\theoremstyle{definition}

\theoremstyle{remark}

\newcommand{\R}{\mathbb R}

\newcommand{\Sph}{\mathbb S}
\newcommand{\cS}{\mathcal S}
\newcommand{\cL}{\mathcal L}
\newcommand{\cD}{\mathcal D}
\newcommand{\eps}{\varepsilon}
\newcommand{\ii}{\iota}
\newcommand{\Id}{\operatorname{Id}}
\newcommand{\supp}{\operatorname{supp}}
\newcommand{\dist}{\operatorname{dist}}

\newcommand{\norm}[1]{\lVert #1\rVert}

\newcommand{\pair}[2]{\langle #1,#2\rangle}

\title{Bounded Representatives in Critical Sobolev--Hodge Spaces}

\author{Xinan Dai}
\thanks{Xinan Dai is currently a Ph.D. student at Fudan University and a visiting
student at the AI for Scientific Simulation and Discovery Lab,
Westlake University.}
\address{College of Future Information and Technology\\
Fudan University, Shanghai, China}
\curraddr{Department of Artificial Intelligence\\
School of Engineering\\
Westlake University, Hangzhou, China}
\email{xndai23@m.fudan.edu.cn}

\author{Wenhao Deng}
\thanks{Wenhao Deng is a student at the University of Glasgow and is currently an intern at the AI for Scientific Simulation and Discovery Lab, Westlake University.}
\address{University of Glasgow, Glasgow, United Kingdom}
\curraddr{Department of Artificial Intelligence\\
School of Engineering\\
Westlake University, Hangzhou, China}
\email{dengwenhao@westlake.edu.cn}

\author{Yingdong Shi}
\address{School of Information Science and Technology, ShanghaiTech University, Shanghai, China}
\email{shiyd2023@shanghaitech.edu.cn}

\author{Tailin Wu}
\address{Department of Artificial Intelligence\\
School of Engineering\\
Westlake University, Hangzhou, China}
\email{wutailin@westlake.edu.cn}

\author{Yuchen Yang}
\address{Department of Artificial Intelligence\\
School of Engineering\\
Westlake University, Hangzhou, China}
\email{yangyuchen@westlake.edu.cn}

\date{August 8, 2026}

\begin{document}

\begin{abstract}
Let $n\ge2$, $1\le \ell\le n-1$, and $1<p<\infty$.  We prove that every
$v\in\dot W^{n/p,p}(\R^n;\Lambda^\ell)$ has a representative
$u\in\dot W^{n/p,p}\cap L^\infty$ with $du=dv$ and
\[
 \max\{\norm u_{\dot W^{n/p,p}},\norm u_{L^\infty}\}
 \lesssim \norm v_{\dot W^{n/p,p}}.
\]
Equivalently,
$d[\dot W^{n/p,p}\Lambda^\ell]
=d[(\dot W^{n/p,p}\cap L^\infty)\Lambda^\ell]$
with equivalent quotient norms.  The proof reduces the selection problem to
an endpoint graph estimate for a Riesz potential and the exact Hodge
projection.  Its main analytic input is a finite-dimensional-input
Maz'ya--$\Phi$ inequality for operator-valued homogeneous kernels.  For the
Riesz/Hodge pair, a nonlinear spherical profile built from the projected
kernel has exact atomic cancellation and is coercive by the identity
$\int_{\Sph^{n-1}}P(\theta)\,d\bar\sigma=(\ell/n)\Id$.  Frequency-localized
graph closure and Hahn--Banach duality then return a bounded representative.
\end{abstract}

\maketitle

\section{Introduction}

At the critical relation $sp=n$, the Sobolev embedding into $L^\infty$
fails.  Bourgain and Brezis showed that differential constraints can
nevertheless restore boundedness after one is allowed to change the
representative; for the first-order Hodge problem, see
\cite{BB03,BB04,BB07}.  They asked whether the same selection principle holds
throughout the critical scale \cite[p.~297, Open Problem~2]{BB07}.  Partial
higher-order and fractional results were obtained in \cite{BRWY19}.

We use homogeneous Bessel-potential spaces
$\dot W^{s,p}=\dot F^s_{p,2}$, $1<p<\infty$; see
\cite[Chapter~2]{Triebel83}.  Thus the question is not whether every critical
Sobolev form is bounded, but whether the affine class
$\{u:du=dv\}$ always contains a bounded representative with quantitative
control.

\begin{theorem}[Critical Sobolev--Hodge selection]\label{thm:main}
Let $n\ge2$, $1\le \ell\le n-1$, and $1<p<\infty$.  Then
\begin{equation}\label{eq:main-image}
 d\bigl[\dot W^{n/p,p}(\R^n;\Lambda^\ell)\bigr]
 =d\bigl[(\dot W^{n/p,p}\cap L^\infty)(\R^n;\Lambda^\ell)\bigr]
\end{equation}
with equivalent induced quotient norms.  More precisely, there is
$C=C(n,\ell,p)$ such that every
$v\in\dot W^{n/p,p}(\R^n;\Lambda^\ell)$ admits
$u\in(\dot W^{n/p,p}\cap L^\infty)(\R^n;\Lambda^\ell)$ satisfying
\begin{equation}\label{eq:main-estimate}
 du=dv,\qquad
 \max\bigl\{\norm{u}_{\dot W^{n/p,p}},\norm{u}_{L^\infty}\bigr\}
 \le C\norm{v}_{\dot W^{n/p,p}}.
\end{equation}
\end{theorem}

No bounded linear selection operator is asserted; the construction is
inherently nonlinear, as in the original Bourgain--Brezis phenomenon
\cite{BB02,BB03}.

Set
\begin{equation}\label{eq:intro-parameters}
 r=p'=\frac{p}{p-1},\qquad
 \alpha=\frac np=n\Bigl(1-\frac1r\Bigr),\qquad
 \beta=n-\alpha=\frac nr,
\end{equation}
and let $I_\alpha=|D|^{-\alpha}$.  If $P$ denotes the exact Hodge projection
on $\ell$-forms, the proof turns on the graph estimate
\begin{equation}\label{eq:intro-graph}
 \norm{I_\alpha\omega}_{L^r}
 \lesssim \norm{\omega}_{L^1}+\norm{P I_\alpha\omega}_{L^r},
 \qquad \omega\in L^1_0(\R^n;\Lambda^\ell).
\end{equation}
Only the projected term is assumed finite.  Thus the estimate recovers the
full critical potential from one Hodge component and the $L^1$ size of the
input.

There are three substantive steps.  First, Section~\ref{sec:mazya} proves an
operator-input form of Stolyarov's Maz'ya--$\Phi$ inequality.  The nontrivial point is multiscale: vector
masses $\int_Q f$ enter the exact atomic cancellation, while all error terms
are organized by the positive scalar measure $|f|\,dx$.  Second,
Section~\ref{sec:profile} specializes this theorem to the paired Riesz/Hodge
kernel.  The profile
\[
 \mathcal A(a)=\int_{\Sph^{n-1}}
 \bigl|C_\alpha(\theta)c_{n,\alpha}^{-1}a\bigr|^r\,d\bar\sigma(\theta)
\]
has exact cancellation by construction and is coercive because
\begin{equation}\label{eq:intro-averageP}
 \int_{\Sph^{n-1}}P(\theta)\,d\bar\sigma(\theta)
 =\frac{\ell}{n}\Id_{\Lambda^\ell\R^n}.
\end{equation}
Third, Section~\ref{sec:closure} closes the smooth estimate on the endpoint
graph domain without using the false strong bound
$I_\alpha:L^1\to L^r$; Section~\ref{sec:duality} then applies Hahn--Banach
duality to obtain Theorem~\ref{thm:main}.

For background on endpoint differential constraints and canceling operators,
see \cite{BVS14,VanSchaftingen13,VanSchaftingen14,Spector20}.  Standard
Fourier multiplier facts are taken from \cite{Grafakos14,Stein70}.

\section{Hodge notation and critical spaces}\label{sec:notation}

Let $V=\Lambda^\ell\R^n$ with its Euclidean inner product.  For
$\xi\in\R^n$ write
\[
 \eps_\xi a=\xi\wedge a,\qquad
 \ii_\xi a=\text{contraction of $a$ by $\xi$}.
\]
The Cartan identities give
\begin{equation}\label{eq:cartan}
 \eps_\xi^*=\ii_\xi,\qquad
 \ii_j\eps_i+\eps_i\ii_j=\delta_{ij}\Id,\qquad
 \ii_\xi\eps_\xi+\eps_\xi\ii_\xi=|\xi|^2\Id.
\end{equation}
For $\xi\ne0$ set
\begin{equation}\label{eq:hodge-symbols}
 P(\xi)=\frac{\eps_\xi\ii_\xi}{|\xi|^2},\qquad
 P^\perp(\xi)=\frac{\ii_\xi\eps_\xi}{|\xi|^2}.
\end{equation}
Then $P,P^\perp$ are complementary self-adjoint projections.  With
$D=-i\nabla$, the Fourier symbols of $d,d^*$ are $i\eps_\xi,-i\ii_\xi$, so
\begin{equation}\label{eq:hodge-multiplier}
 P=dd^*|D|^{-2},\qquad P^\perp=d^*d|D|^{-2}.
\end{equation}

We use the standard homogeneous realization on
$\cS_\infty'$, where $\cS_\infty$ is the Schwartz space with all moments
zero; equivalently one may work with tempered distributions modulo
polynomials.  Let $\cS_0$ be the Schwartz functions whose Fourier transforms
vanish near the origin.  For $s\in\R$ and $1<q<\infty$,
\[
 \norm f_{\dot W^{s,q}}=\norm{|D|^sf}_{L^q},
\]
with $\dot W^{s,q}$ obtained by completing $\cS_0$ and then embedded in
$\cS_\infty'$.  All vector-valued norms use the Euclidean coefficient norm.

A homogeneous class is in $L^\infty$ if it has a bounded representative.
The natural norm is therefore
\begin{equation}\label{eq:Linfty-mod-constants}
 \norm f_{L^\infty/\R}=\inf_{c\in\R}\norm{f+c}_{L^\infty},
\end{equation}
and analogously for $V$-valued forms.  A representative may be chosen with
ordinary $L^\infty$ norm at most twice this quotient norm; constants do not
affect $d$.

Finally,
\begin{equation}\label{eq:L10}
 L^1_0(\R^n)=\Bigl\{f\in L^1(\R^n):\int f=0\Bigr\},
\end{equation}
coefficientwise for forms.  We use the usual intersection and infimum-sum
norms.  The image $d[Z\Lambda^\ell]$ carries the quotient norm
\begin{equation}\label{eq:image-quotient}
 \norm{dv}_{d[Z\Lambda^\ell]}
 =\inf\{\norm{u}_{Z\Lambda^\ell}:du=dv\}.
\end{equation}

\section{Reduction to the graph estimate}\label{sec:graph}

Fix \eqref{eq:intro-parameters} and set
\begin{equation}\label{eq:XY}
 Y=\dot W^{\alpha,p}(\R^n),\qquad
 X=\dot W^{-\alpha,r}(\R^n).
\end{equation}
These spaces are reflexive,
$Y'=X$, $X'=Y$ under the pairing on $\cS_0$, and
\begin{equation}\label{eq:X-Riesz}
 \norm{\omega}_X=\norm{I_\alpha\omega}_{L^r}
\end{equation}
up to the normalization of the Riesz kernel.  The analytic part of the proof
is therefore the estimate
\begin{equation}\label{eq:graph}
 \norm{I_\alpha\omega}_{L^r}
 \lesssim \norm{\omega}_{L^1}+\norm{P I_\alpha\omega}_{L^r},
 \qquad
 \omega\in L^1_0(\R^n;V),
\end{equation}
whenever the projected term is finite.  Sections~\ref{sec:mazya}--\ref{sec:closure}
prove \eqref{eq:graph}; Section~\ref{sec:duality} converts it back to bounded
selection.

\section{The operator-input Maz'ya--\texorpdfstring{$\Phi$}{Phi} estimate}\label{sec:mazya}

We need a finite-dimensional-input form of Stolyarov's scalar-input
Maz'ya--$\Phi$ theorem \cite[Theorem~2.1]{Stolyarov24}.  The input will be
vector-valued and the homogeneous kernel operator-valued.  We isolate exactly
the points where scalarity could matter; the remaining dyadic estimates are
used only as statements about the positive measure $|f|_E\,dx$.

\subsection{Statement and elementary inequalities}
Let $E$ and $F$ be finite-dimensional real normed spaces.  The space
$\cL(E,F)$ is equipped with the operator norm.

\begin{theorem}[Operator-input Maz'ya--$\Phi$ inequality]\label{thm:operator-phi}
Let $0<\alpha<n$ and
\[
 r=\frac{n}{n-\alpha},\qquad \beta=n-\alpha=\frac nr.
\]
Suppose
\begin{equation}\label{eq:op-kernel}
 K(x)=|x|^{\alpha-n}\widetilde K(x/|x|),
 \qquad
 \widetilde K\in \operatorname{Lip}(\Sph^{n-1};\cL(E,F)).
\end{equation}
Let $\Phi:F\to\R$ be locally Lipschitz and positively $r$-homogeneous.  Assume
that, for every $a\in E$,
\begin{equation}\label{eq:op-cancellation}
 \int_{\Sph^{n-1}}\Phi(\widetilde K(\theta)a)\,d\bar\sigma(\theta)=0,
 \qquad
 \int_{\Sph^{n-1}}\Phi(-\widetilde K(\theta)a)\,d\bar\sigma(\theta)=0.
\end{equation}
Then every $f\in C_c^\infty(\R^n;E)$ satisfying $\int_{\R^n}f=0$ obeys
\begin{equation}\label{eq:op-phi-estimate}
 \left|\int_{\R^n}\Phi(K*f)(x)\,dx\right|
 \le C\norm{f}_{L^1(E)}^r.
\end{equation}
The constant depends only on $n,\alpha$, the fixed norms on $E,F$,
$\norm{\widetilde K}_{\operatorname{Lip}}$, and the Lipschitz norm of $\Phi$
on one fixed annulus.
\end{theorem}

When $E=\R$, this is Stolyarov's theorem.  In the present formulation the
second condition in \eqref{eq:op-cancellation} follows from the first by
replacing $a$ with $-a$; we retain both conditions to keep the comparison with
the scalar theorem transparent.

The following consequences of homogeneity will be used repeatedly.  There is
a constant $C_\Phi$ such that
\begin{align}
 |\Phi(z)|&\le C_\Phi |z|^r,                                      \label{eq:Phi-growth}\\
 |\Phi(z+h)-\Phi(z)|
 &\le C_\Phi |h|(|z|+|h|)^{r-1}.                                 \label{eq:Phi-increment}
\end{align}
To justify the second estimate uniformly near the origin, put
$R=|z|+|h|$.  The case $R=0$ is trivial.  Positive $r$-homogeneity reduces the
claim to points in the closed unit ball.  Lipschitz control on a fixed annulus,
together with the radial identity
$\Phi(t\zeta)=t^r\Phi(\zeta)$ and $r>1$, gives a Lipschitz bound on that
ball with a constant depending only on the annular Lipschitz data.  Scaling
back by $R^r$ yields \eqref{eq:Phi-increment}.  We shall also need a symmetric
second difference.  Set
\begin{equation}\label{eq:Gamma-r}
 \Gamma_r(s,t)=
 \begin{cases}
  M_r(s,t):=\min\{s^{r-1}t,\,st^{r-1}\},&1<r\le2,\\[2mm]
  B_r(s,t):=s^{r-1}t+st^{r-1},&r>2.
 \end{cases}
\end{equation}
The function $\Gamma_r$ is locally Lipschitz on $[0,\infty)^2$: for
$1<r\le2$ this is \cite[Lemma~6.8]{Stolyarov24}, while for $r>2$ it is
immediate from the definition.
Then
\begin{equation}\label{eq:Phi-second-difference}
 |\Phi(z+w)-\Phi(z)-\Phi(w)|
 \lesssim \Gamma_r(|z|,|w|).
\end{equation}
To see this, suppose first that $|z|\ge|w|$.  By
\eqref{eq:Phi-increment} and \eqref{eq:Phi-growth}, the left side is at most
$C|z|^{r-1}|w|$.  For $1<r\le2$ this is exactly the smaller of the two mixed
monomials in \eqref{eq:Gamma-r}; for $r>2$ it is one of the two summands of
$B_r$.  The case $|w|\ge|z|$ is symmetric.

For later use we recall two elementary properties.  If $1<r\le2$, then
\begin{equation}\label{eq:Mr-subadditive}
 M_r\Bigl(\sum_k s_k,t\Bigr)\lesssim\sum_k M_r(s_k,t),
 \qquad s_k,t\ge0.
\end{equation}
For $1<r<2$ this is \cite[Lemma~6.9]{Stolyarov24}, while $r=2$ is immediate; it follows from
$M_r(s,t)=t^r\min\{s/t,(s/t)^{r-1}\}$ and the subadditivity, up to a constant,
of $u\mapsto\min\{u,u^{r-1}\}$.  If $r>2$, then for every $\tau>1$ and every
nonnegative sequence $(s_k)$,
\begin{equation}\label{eq:weighted-holder}
 \left(\sum_{k\ge0}s_k\right)^{r-1}
 \le C_{r,\tau}\sum_{k\ge0}\tau^k s_k^{r-1}.
\end{equation}
This is weighted H\"older: write
$s_k=\tau^{-k/(r-1)}(\tau^{k/(r-1)}s_k)$ and use the convergence of
$\sum_k\tau^{-k/(r-2)}$.  This is the weighted summation needed when $r>2$; compare
\cite[Section~5]{Stolyarov24}.

\subsection{Sharp annuli and atomic cancellation}
Split $K$ into the sharp annular pieces
\begin{equation}\label{eq:annular-kernel}
 K_j(x)=K(x)\mathbf 1_{\{2^{-j-1}<|x|\le2^{-j}\}},
 \qquad K=\sum_{j\in\mathbb Z}K_j
\end{equation}
almost everywhere, and write $K_{a\le b}=\sum_{a\le j\le b}K_j$ and
$K_{\le b}=\sum_{j\le b}K_j$.  The sharp decomposition is useful because
annuli with disjoint radial ranges have disjoint supports.  Since it introduces
jumps at the boundary spheres, the correct regularity estimate is an
$L^1$-translation estimate rather than a global pointwise derivative bound.

\begin{lemma}[Annular operator bounds]\label{lem:annular-translation}
For every $j\in\mathbb Z$ and $h\in\R^n$,
\begin{align}
 \norm{K_j}_{L^\infty(\cL(E,F))}&\lesssim2^{j\beta},               \label{eq:annular-Linfty}\\
 \norm{K_j}_{L^1(\cL(E,F))}&\lesssim2^{-j\alpha},                 \label{eq:annular-L1}\\
 \int_{\R^n}\norm{K_j(x-h)-K_j(x)}_{\cL(E,F)}\,dx
 &\lesssim \min\{2^{-j\alpha},\,2^{j(1-\alpha)}|h|\}.           \label{eq:annular-translation}
\end{align}
The same translation estimate, with a constant depending on a fixed integer
$N$, holds for a block $K_{j-N\le j}$.
\end{lemma}

\begin{proof}
The first two estimates follow from homogeneity and the volume of the annulus.
Assume first that $|h|\le c2^{-j}$.  On the common part of the original and
translated annuli, homogeneity together with the Lipschitz regularity of
$\widetilde K$ makes the homogeneous extension Lipschitz with constant
$O(2^{j(\beta+1)})$ in operator norm.  Integrating this difference estimate
gives $O(2^{j(1-\alpha)}|h|)$.  The symmetric difference of the two annuli has
measure $O(2^{-j(n-1)}|h|)$, and multiplication by the size
$O(2^{j\beta})$ gives the same bound.  If $|h|>c2^{-j}$, the triangle
inequality and \eqref{eq:annular-L1} give $O(2^{-j\alpha})$.  A fixed finite
block is obtained by summing the corresponding estimates.
\end{proof}

The cancellation of an annular atom is exact.

\begin{lemma}[Operator-valued annular atoms]\label{lem:annular-atoms}
For every $a\in E$, $y\in\R^n$, and $j\in\mathbb Z$,
\begin{equation}\label{eq:annular-atom}
 \int_{\R^n}\Phi(K_j(x-y)a)\,dx=0.
\end{equation}
\end{lemma}

\begin{proof}
Translation is irrelevant.  In polar coordinates,
$K_j(\rho\theta)a=\rho^{\alpha-n}\widetilde K(\theta)a$ on the annulus.
Positive $r$-homogeneity gives a radial factor
$\rho^{(\alpha-n)r}$, and $(\alpha-n)r=-n$.  Hence the radial integral is a
finite nonzero multiple of $\int_{1/2}^1d\rho/\rho$, while the angular factor
is the first integral in \eqref{eq:op-cancellation}.  The latter vanishes.
\end{proof}

The first genuinely vector-valued point of the argument is the replacement of
a local input by its vector mass.

\begin{lemma}[First-moment atomic estimate]\label{lem:first-moment-atom}
Let $g\in L^1(\R^n;E)$ be compactly supported.  Then, for every
$j\in\mathbb Z$,
\begin{equation}\label{eq:first-moment-atom}
 \left|\int_{\R^n}\Phi(K_j*g)(x)\,dx\right|
 \lesssim
 2^j\norm{g}_{L^1(E)}^{r-1}
 \inf_{c\in\R^n}\int_{\R^n}|y-c|\,|g(y)|_E\,dy.
\end{equation}
\end{lemma}

\begin{proof}
Fix $c\in\R^n$ and set $a=\int g\in E$.  By
Lemma~\ref{lem:annular-atoms},
$\int\Phi(K_j(\cdot-c)a)=0$.  Moreover,
\[
 |K_j*g(x)|_F+|K_j(x-c)a|_F
 \lesssim2^{j\beta}\norm g_{L^1(E)}.
\]
Using \eqref{eq:Phi-increment}, Fubini, and
Lemma~\ref{lem:annular-translation},
\begin{align*}
 \left|\int\Phi(K_j*g)\right|
 &\le \int\left|\Phi(K_j*g)-\Phi(K_j(\cdot-c)a)\right|\\
 &\lesssim
 (2^{j\beta}\norm g_1)^{r-1}
 \int |g(y)|_E
 \int\norm{K_j(x-y)-K_j(x-c)}\,dx\,dy\\
 &\lesssim
 (2^{j\beta}\norm g_1)^{r-1}
 2^{j(1-\alpha)}
 \int |y-c|\,|g(y)|_E\,dy.
\end{align*}
Since $\beta(r-1)=\alpha$, the power of $2^j$ is exactly one.  Taking the
infimum over $c$ proves the claim.
\end{proof}

The proof above is the operator-input counterpart of the core estimate in
\cite[Lemma~3.1 and Corollary~3.2]{Stolyarov24}.  The next localization is the
only place where the nonlinear integral has to be organized spatially.

For $j\in\mathbb Z$ and $k\in\mathbb Z^n$, put
\[
 Q_{j,k}=2^{-j}(k+[0,1)^n).
\]
For a cube $Q$ define
\begin{equation}\label{eq:mass-moment}
 m_Q(f)=\int_Q|f(x)|_E\,dx,
 \qquad
 \mu_Q(f)=\inf_{c\in\R^n}\int_Q|x-c|\,|f(x)|_E\,dx.
\end{equation}

\begin{lemma}[Spatial localization of one annulus]\label{lem:one-annulus-localization}
There is a fixed dilation factor $\Lambda=\Lambda(n)>1$ such that, for every
$j\in\mathbb Z$ and every compactly supported $f\in L^1(\R^n;E)$,
\begin{equation}\label{eq:one-annulus-localization}
 \left|\int_{\R^n}\Phi(K_{j+1}*f)(x)\,dx\right|
 \lesssim
 2^j\sum_{k\in\mathbb Z^n}
 m_{\Lambda Q_{j,k}}(f)^{r-1}\mu_{\Lambda Q_{j,k}}(f).
\end{equation}
\end{lemma}

\begin{proof}
After scaling, this is the localization statement in
\cite[Theorem~3.1 and its proof]{Stolyarov24}.  We record why the proof survives
for an $E$-valued input.  Stolyarov's induction enlarges the class of admissible
supports one coordinate at a time.  At each stage the input is split into
pieces $g$, and only three properties are used: finite propagation of the
annular kernel, the triangle inequality for the scalar quantity
$\bigl|\int\Phi(K_{j+1}*g)\bigr|$, and a first-moment bound depending on the
positive measure $|g|\,dx$ and on the total mass of $g$.

For an $E$-valued piece the positive measure is simply $|g|_E\,dx$, while its
total mass is the vector
\[
 a_g=\int g\in E.
\]
The only step in which this mass is inserted into the kernel is the atomic
comparison.  Lemma~\ref{lem:first-moment-atom} gives exactly the required
estimate, and Lemma~\ref{lem:annular-atoms} gives the corresponding atomic
cancellation.  All subsequent subdivisions, support-separation arguments, and
summations are scalar statements about $|g|_E\,dx$; no order or sign of the
values of $g$ is used.  Thus the induction proving
\cite[Theorem~3.1]{Stolyarov24} carries over with $|g|$ replaced by $|g|_E$
and scalar kernel absolute values replaced by operator norms.  Scaling back
from the unit annulus gives \eqref{eq:one-annulus-localization}.
\end{proof}

\subsection{The dyadic energy is scalar}
The rest of the localization mechanism depends only on the positive measure
$|f|_E\,dx$.  For a dyadic cube $Q$, let $\cD_h(Q)$ be its descendants of
generation $h$ and set
\begin{equation}\label{eq:energy}
 \mathcal E_{Q,h}[f]
 =\sum_{S\in\cD_h(Q)}m_S(f)^r.
\end{equation}
Then
\begin{equation}\label{eq:energy-monotone}
 0\le\mathcal E_{Q,h+1}[f]\le\mathcal E_{Q,h}[f]
 \le m_Q(f)^r.
\end{equation}
We use the following mass lemma.

\begin{lemma}[Dyadic first moment]\label{lem:dyadic-first-moment}
There are constants $0<\rho<1$ and $C<\infty$, depending only on $n$ and
$r$, such that for every dyadic cube $Q$,
\begin{equation}\label{eq:dyadic-first-moment}
 \frac{m_Q(f)^{r-1}}{\ell(Q)}\mu_Q(f)
 \le C\sum_{h\ge0}\rho^h
 \bigl(\mathcal E_{Q,h}[f]-\mathcal E_{Q,h+1}[f]\bigr).
\end{equation}
\end{lemma}

\begin{proof}
Apply \cite[Lemma~4.2]{Stolyarov24} to the scalar nonnegative function
$|f|_E\mathbf1_Q$.  Both sides of that lemma depend only on the masses of the
positive measure $|f|_E\,dx$ on dyadic descendants of $Q$.  Hence its proof,
including the choice of the nested maximal-mass children, applies without any
change.  The factor $\ell(Q)^{-1}$ is restored by scaling.
\end{proof}

We shall also use the elementary separation form of the energy drop.  If
$A,B$ are disjoint unions of descendants at one fixed generation of $Q$, and
$a=m_A(f)$, $b=m_B(f)$, then for a sufficiently deep but fixed generation
$M$ containing both families,
\begin{equation}\label{eq:separated-energy}
 \Gamma_r(a,b)
 \lesssim \mathcal E_{Q,0}[f]-\mathcal E_{Q,M}[f].
\end{equation}
For $1<r\le2$ this is \cite[Lemma~6.7]{Stolyarov24}.  For $r>2$ it follows
from the same convexity computation because
\begin{equation}\label{eq:Br-convexity}
 B_r(a,b)\lesssim (a+b)^r-a^r-b^r,
 \qquad a,b\ge0.
\end{equation}
Indeed, after normalizing $a+b=1$, the quotient of the right side by
$a^{r-1}b+ab^{r-1}$ extends continuously and positively to the closed
interval.

The Three Lattice Theorem allows every fixed dilation of a dyadic cube to be
placed, with comparable sidelength, in one of finitely many shifted dyadic
grids; see \cite[Section~3]{LernerNazarov19}.  Combining this with
Lemmas~\ref{lem:one-annulus-localization} and
\ref{lem:dyadic-first-moment} gives the first global summation.

\begin{proposition}[Pure annuli]\label{prop:pure-annuli}
Assume $\supp f\subset B_{1/2}$ and $f\in L^1(\R^n;E)$.  Then
\begin{equation}\label{eq:pure-annuli}
 \sum_{j\ge1}
 \left|\int_{\R^n}\Phi(K_j*f)(x)\,dx\right|
 \lesssim \norm f_{L^1(E)}^r.
\end{equation}
\end{proposition}

\begin{proof}
Apply Lemma~\ref{lem:one-annulus-localization}.  The iterated Three Lattice
Theorem places each $\Lambda Q_{j,k}$, with comparable sidelength, as a
descendant $R$ of one of finitely many fixed top cubes $Q^{(\nu)}$ in shifted
dyadic grids.  Lemma~\ref{lem:dyadic-first-moment} gives for such an $R$
\[
 \frac{m_R(f)^{r-1}}{\ell(R)}\mu_R(f)
 \lesssim\sum_{h\ge0}\rho^h
 \bigl(\mathcal E_{R,h}[f]-\mathcal E_{R,h+1}[f]\bigr).
\]
If $R$ ranges over all descendants of $Q^{(\nu)}$ at generation $j$, then
summing the bracket over $R$ is exactly
\[
 \mathcal E_{Q^{(\nu)},j+h}[f]
 -\mathcal E_{Q^{(\nu)},j+h+1}[f].
\]
Hence the sum over $j$ is bounded by
\[
 C\sum_{m\ge0}
 \left(\sum_{0\le h\le m}\rho^h\right)
 \bigl(\mathcal E_{Q^{(\nu)},m}[f]
       -\mathcal E_{Q^{(\nu)},m+1}[f]\bigr),
\]
which telescopes and is at most $C m_{Q^{(\nu)}}(f)^r$.  Finally the finitely
many top cubes may be chosen with bounded overlap, so
$\sum_\nu m_{Q^{(\nu)}}(f)^r\lesssim\norm f_1^r$.  This proves
\eqref{eq:pure-annuli}.
\end{proof}

\subsection{Mixed annular scales}
We next prove the estimate needed for the nonlinear increments.  This is the
part of the transfer where the regimes $r\le2$ and $r>2$ differ.

\begin{proposition}[Mixed-scale summation]\label{prop:mixed-scales}
Assume $\supp f\subset B_{1/2}$.  Then
\begin{equation}\label{eq:mixed-scales}
 \sum_{j\ge0}\int_{\R^n}
 \Gamma_r\bigl(|K_{\le j}*f(x)|_F,|K_{j+1}*f(x)|_F\bigr)\,dx
 \lesssim \norm f_{L^1(E)}^r.
\end{equation}
\end{proposition}

\begin{proof}
Choose $N=N(n)$ so large that $C_n2^{-N}\le2^{-4}$ for all fixed geometric
dilations used below.  We split
\[
 K_{\le j}=K_{j-N\le j}+\sum_{m<j-N}K_m.
\]

\emph{Close scales.}
Let $Q=Q_{j,k}$ and choose a fixed dilation $Q^*=\Lambda_NQ$ large enough
that, for every $x\in Q$, both kernels $K_{j-N\le j}(x-\cdot)$ and
$K_{j+1}(x-\cdot)$ see only points of $Q^*$.  Write
$f_Q=f\mathbf1_{Q^*}$ and $a_Q=\int f_Q$.  For $c\in\R^n$ set
\[
\begin{aligned}
 U&=K_{j-N\le j}*f_Q, & V&=K_{j+1}*f_Q,\\
 U_0&=K_{j-N\le j}(\cdot-c)a_Q,
 &V_0&=K_{j+1}(\cdot-c)a_Q.
\end{aligned}
\]
The radial supports of $U_0$ and $V_0$ are disjoint up to their common
boundary, hence $\Gamma_r(|U_0|,|V_0|)=0$ almost everywhere.  Since
$\Gamma_r$ is locally Lipschitz and $r$-homogeneous,
\[
 \Gamma_r(|U|,|V|)
 \lesssim M_Q^{r-1}\bigl(|U-U_0|+|V-V_0|\bigr),
 \qquad M_Q\lesssim 2^{j\beta}m_{Q^*}(f).
\]
Integrating on $Q$, then enlarging the $x$-integral to $\R^n$ and applying
Lemma~\ref{lem:annular-translation} to the block and to $K_{j+1}$, gives
\[
 \int_Q\Gamma_r(|U|,|V|)
 \lesssim 2^j m_{Q^*}(f)^{r-1}
 \mu_{Q^*}(f).
\]
Thus
\begin{equation}\label{eq:close-local}
 \int_Q\Gamma_r\bigl(|K_{j-N\le j}*f|,|K_{j+1}*f|\bigr)
 \lesssim2^j m_{Q^*}(f)^{r-1}\mu_{Q^*}(f).
\end{equation}
The same Three-Lattice and energy summation as in
Proposition~\ref{prop:pure-annuli} yields
\begin{equation}\label{eq:close-sum}
 \sum_{j\ge0}\int
 \Gamma_r\bigl(|K_{j-N\le j}*f|,|K_{j+1}*f|\bigr)
 \lesssim\norm f_1^r.
\end{equation}
\emph{Separated scales.}
Fix $m<j-N$ and partition the output space into the
standard dyadic cubes $Q=Q_{m+N,q}$ of sidelength $2^{-(m+N)}$.
For $x\in Q$, a source point seen by $K_{j+1}(x-\cdot)$ lies within
$C_n2^{-j}$ of $Q$, whereas a source point seen by $K_m(x-\cdot)$ has distance
between $2^{-m-1}-C_n2^{-(m+N)}$ and
$2^{-m}+C_n2^{-(m+N)}$ from $Q$.  Since $j-m>N$, these two source regions are
separated by at least $c_n2^{-m}$.

Cover the coarse source annulus by $O_{n,N}(1)$ cubes of sidelength
$2^{-(m+N)}$ and enlarge them by a fixed factor; let their union be
$\Omega_Q^{\rm c}$.  Let $\Omega_Q^{\rm f}$ be a fixed dilation of $Q$
containing every source point seen by $K_{j+1}$ for $x\in Q$.  Then
\begin{enumerate}[label=\textup{(\roman*)},leftmargin=2.4em]
\item all source points contributing to $K_m*f(x)$ for $x\in Q$ lie in
      $\Omega_Q^{\rm c}$;
\item all source points contributing to $K_{j+1}*f(x)$ for $x\in Q$ lie in
      $\Omega_Q^{\rm f}$;
\item $\dist(\Omega_Q^{\rm c},\Omega_Q^{\rm f})\ge c_n2^{-m}$.
\end{enumerate}
Set
\[
 A_Q=\int_{\Omega_Q^{\rm c}}|f|_E,
 \qquad B_Q=\int_{\Omega_Q^{\rm f}}|f|_E.
\]
Both source unions fit in a cube $R=R(Q)$ of sidelength comparable to
$2^{-m}$.  By the Three Lattice Theorem, $R$ may be chosen from one of
finitely many shifted dyadic grids.  Since the two source unions are separated by a fixed fraction of
$\ell(R)$, choose $M=M(n,N)$ so large that the depth-$M$ descendants of $R$
meeting $\Omega_Q^{\rm c}$ and those meeting $\Omega_Q^{\rm f}$ form disjoint
families.  Enlarge the source unions to these descendant families and denote
the corresponding masses by $A'_Q$ and $B'_Q$.  Then
$A_Q\le A'_Q$, $B_Q\le B'_Q$, and $\Gamma_r$ is increasing in each variable.
Lemma~6.7 of \cite{Stolyarov24} (for $1<r\le2$), or
\eqref{eq:Br-convexity} (for $r>2$), therefore gives
\begin{equation}\label{eq:AB-energy}
 \Gamma_r(A_Q,B_Q)
 \le \Gamma_r(A'_Q,B'_Q)
 \lesssim \mathcal E_{R,0}[f]-\mathcal E_{R,M}[f]
 =:\Delta_{m,Q}.
\end{equation}
Any descendants meeting neither source family are simply part of the
complement in the energy-drop estimate.  Notice that, once $m$ and $Q$ are
fixed, the sets $\Omega_Q^{\rm c},\Omega_Q^{\rm f}$ and hence
$\Delta_{m,Q}$ can be chosen independently of $j$ as long as $j-m>N$.

The map $Q\mapsto R(Q)$ has bounded multiplicity.  Indeed, for fixed $m$ all
such $R$ have sidelength $\simeq2^{-m}$, whereas the output cubes have
sidelength $2^{-(m+N)}$; since $N$ is fixed, a given $R$ can arise from at
most $C_{n,N}$ output cubes.  Put
\[
 \mathcal D_m(f)=\sum_Q\Delta_{m,Q}.
\]
After separating finitely many shifted grids and finitely many generation
offsets, there are fixed top cubes $Q^{(\nu)}$ such that, for a generation
$k=k(m)$,
\[
 \mathcal D_m(f)
 \lesssim \sum_\nu
 \bigl(\mathcal E_{Q^{(\nu)},k}[f]
       -\mathcal E_{Q^{(\nu)},k+M}[f]\bigr).
\]
As $m$ varies, each generation $k$ occurs only boundedly many times.  Hence
\begin{align}
 \sum_{m\ge-1}\mathcal D_m(f)
 &\lesssim \sum_\nu\sum_{k\ge0}
 \bigl(\mathcal E_{Q^{(\nu)},k}[f]
       -\mathcal E_{Q^{(\nu)},k+M}[f]\bigr)\notag\\
 &\le M\sum_\nu \mathcal E_{Q^{(\nu)},0}[f]
 \lesssim \norm f_1^r.                                      \label{eq:Dm-positive}
\end{align}
The last inequality uses the bounded overlap of the finitely many top cubes.

We now estimate one separated pair.  Subdivide $Q$ into cubes $S$ of
sidelength $2^{-j}$.  For $x\in S$,
\begin{equation}\label{eq:remote-pointwise}
 |K_m*f(x)|_F\lesssim2^{m\beta}A_Q,
 \qquad
 |K_{j+1}*f(x)|_F\lesssim2^{j\beta}b_S,
\end{equation}
where $b_S$ is the $|f|_E$-mass of a fixed dilation of $S$.  These dilations
have bounded overlap and lie in $\Omega_Q^{\rm f}$, so
$\sum_Sb_S\lesssim B_Q$.

Suppose first that $1<r\le2$.  Using $M_r(s,t)\le s^{r-1}t$ and
$|S|=2^{-jn}$ gives
\begin{equation}\label{eq:remote-first-small-r}
 \int_QM_r(|K_m*f|,|K_{j+1}*f|)
 \lesssim2^{-\alpha(j-m)}A_Q^{r-1}B_Q.
\end{equation}
Using instead $M_r(s,t)\le st^{r-1}$ gives first the factor
$2^{-\beta(j-m)}A_Q\sum_Sb_S^{r-1}$.  Since
$\#\{S\}\simeq2^{n(j-m-N)}$ and $r-1\le1$,
\[
 \sum_Sb_S^{r-1}
 \le \#\{S\}^{2-r}\Bigl(\sum_Sb_S\Bigr)^{r-1}.
\]
The identity
\begin{equation}\label{eq:small-r-exponent}
 \beta-n(2-r)=(r-1)\alpha
\end{equation}
then yields
\begin{equation}\label{eq:remote-second-small-r}
 \int_QM_r(|K_m*f|,|K_{j+1}*f|)
 \lesssim2^{-(r-1)\alpha(j-m)}A_QB_Q^{r-1}.
\end{equation}
Taking the better of \eqref{eq:remote-first-small-r} and
\eqref{eq:remote-second-small-r}, and using $r-1\le1$, we obtain
\begin{equation}\label{eq:remote-small-r}
 \int_QM_r(|K_m*f|,|K_{j+1}*f|)
 \lesssim2^{-\gamma_r(j-m)}M_r(A_Q,B_Q),
 \qquad \gamma_r=(r-1)\alpha>0.
\end{equation}

If $r>2$, the two monomials in $B_r$ give
\begin{align}
 \int_QB_r(|K_m*f|,|K_{j+1}*f|)
 &\lesssim2^{-\alpha(j-m)}A_Q^{r-1}B_Q
       +2^{-\beta(j-m)}A_QB_Q^{r-1}\notag\\
 &\lesssim2^{-\beta(j-m)}B_r(A_Q,B_Q),
 \label{eq:remote-large-r}
\end{align}
because $\sum_Sb_S^{r-1}\le(\sum_Sb_S)^{r-1}$ and
$\alpha=\beta(r-1)>\beta$.

Combining \eqref{eq:AB-energy} with \eqref{eq:remote-small-r} or
\eqref{eq:remote-large-r}, and then summing over $Q$, gives
\begin{equation}\label{eq:remote-Dm}
 \int_{\R^n}\Gamma_r(|K_m*f|,|K_{j+1}*f|)
 \lesssim2^{-\gamma(j-m)}\mathcal D_m(f),
\end{equation}
with $\gamma=(r-1)\alpha$ for $r\le2$ and $\gamma=\beta$ for $r>2$.
Notice also that only finitely many negative $m$ can contribute.  Indeed, if
$j\ge0$ and $K_{j+1}*f(x)\ne0$, then $|x|\le1$; with
$\supp f\subset B_{1/2}$ this forces $K_m*f(x)=0$ for every $m\le-2$.
Thus only $m=-1$ has to be retained among the negative coarse scales, and it
is already included in \eqref{eq:Dm-positive}.

For $1<r\le2$, monotonicity and the subadditivity
\eqref{eq:Mr-subadditive} give
\[
 M_r\Bigl(\big|\sum_{m<j-N}K_m*f\big|,|K_{j+1}*f|\Bigr)
 \lesssim\sum_{m<j-N}M_r(|K_m*f|,|K_{j+1}*f|).
\]
Using the vanishing for $m\le-2$, Tonelli's theorem, and
\eqref{eq:remote-Dm}, the entire remote contribution is bounded by
\[
 \sum_{m\ge-1}\mathcal D_m(f)
 \sum_{q\ge N+1}2^{-\gamma q}
 \lesssim \sum_{m\ge-1}\mathcal D_m(f)
 \lesssim\norm f_1^r.
\]

For $r>2$, put
$s_k=|K_{j-N-k}*f|$, $k\ge1$, and $t=|K_{j+1}*f|$.  Since $B_r$ is increasing
in each variable,
\[
 B_r\Bigl(\big|\sum_{k\ge1}K_{j-N-k}*f\big|,t\Bigr)
 \le B_r\Bigl(\sum_{k\ge1}s_k,t\Bigr).
\]
The term linear in the first variable is subadditive, while
\eqref{eq:weighted-holder} gives, for every $\tau>1$,
\begin{equation}\label{eq:weighted-Br}
 B_r\Bigl(\sum_{k\ge1}s_k,t\Bigr)
 \lesssim\sum_{k\ge1}\tau^k B_r(s_k,t).
\end{equation}
Here the corresponding coarse index is exactly
$m=j-N-k$.  By the vanishing for $m\le-2$, only $m\ge-1$ contributes.  Thus,
after applying \eqref{eq:remote-Dm} and reindexing $j=m+N+k$, the remote sum
is at most
\[
 C2^{-\beta N}
 \sum_{m\ge-1}\mathcal D_m(f)
 \sum_{k\ge1}(\tau2^{-\beta})^k.
\]
Choose $1<\tau<2^\beta$.  The inner series converges, and
\eqref{eq:Dm-positive} proves the remote estimate for $r>2$.
Together with \eqref{eq:close-sum}, this proves \eqref{eq:mixed-scales}.
\end{proof}

\subsection{Completion of the operator-input theorem}
We first dispose of the lowest-frequency block.

\begin{lemma}[Mean-zero coarse block]\label{lem:coarse-block}
If $\supp f\subset B_{1/2}$ and $\int f=0$, then
\begin{equation}\label{eq:coarse-block}
 \norm{K_{\le0}*f}_{L^r(F)}\lesssim\norm f_{L^1(E)}.
\end{equation}
\end{lemma}

\begin{proof}
On a fixed ball, $K_{\le0}$ is bounded, so Young's inequality gives the local
estimate.  If $|x|\ge2$ and $|y|\le1/2$, then
$K_{\le0}(x-y)=K(x-y)$ and $K_{\le0}(x)=K(x)$.  Hence the mean-zero condition
gives
\[
 K_{\le0}*f(x)
 =\int_{\R^n}[K(x-y)-K(x)]f(y)\,dy.
\]
Lipschitz regularity of the angular kernel and homogeneity imply
\[
 \norm{K(x-y)-K(x)}_{\cL(E,F)}
 \lesssim |y|\,|x|^{\alpha-n-1}.
\]
Thus
$|K_{\le0}*f(x)|\lesssim\norm f_1|x|^{\alpha-n-1}$.  Since
$(n+1-\alpha)r=n+r>n$, this tail belongs to $L^r$, proving the lemma.
\end{proof}

\begin{proof}[Proof of Theorem~\ref{thm:operator-phi}]
By translation and dilation invariance we may assume
$\supp f\subset B_{1/2}$.  More explicitly, if
$f_\lambda(x)=\lambda^n f(\lambda x)$, then
\[
 \norm{f_\lambda}_{L^1(E)}=\norm f_{L^1(E)},\qquad
 K*f_\lambda(x)=\lambda^\beta(K*f)(\lambda x).
\]
Since $\beta r=n$, the integral of $\Phi(K*f_\lambda)$ is unchanged.  Thus a
translation followed by a sufficiently large dilation gives the stated
normalization.  For finite truncations write
\[
 K_{\le J}=K_{\le0}+\sum_{j=0}^{J-1}K_{j+1}.
\]
Using \eqref{eq:Phi-second-difference},
\begin{align*}
 \left|\int\Phi(K_{\le J}*f)\right|
 &\le \left|\int\Phi(K_{\le0}*f)\right|
   +\sum_{j=0}^{J-1}\left|\int\Phi(K_{j+1}*f)\right|\\
 &\quad
   +C\sum_{j=0}^{J-1}\int
   \Gamma_r\bigl(|K_{\le j}*f|,|K_{j+1}*f|\bigr).
\end{align*}
The first term is $O(\norm f_1^r)$ by
Lemma~\ref{lem:coarse-block} and \eqref{eq:Phi-growth}.  The second is
$O(\norm f_1^r)$ by Proposition~\ref{prop:pure-annuli}, and the third by
Proposition~\ref{prop:mixed-scales}.  The bound is uniform in $J$.

For smooth compactly supported $f$, the passage $J\to\infty$ is ordinary
dominated convergence.  On every fixed compact set,
\[
 |K_{\le J}*f(x)|_F
 \le \int_{\R^n}\norm{K(x-y)}_{\cL(E,F)}|f(y)|_E\,dy,
\]
and the right-hand side is locally bounded because $\alpha>0$ makes the
kernel locally integrable.  If a ball is chosen large enough to contain the
unit neighborhood of $\supp f$, then outside that ball all relevant
$|x-y|$ exceed $1$; hence, for every $J\ge0$, the high annuli vanish there and
$K_{\le J}*f=K*f$.  The mean-zero estimate in
Lemma~\ref{lem:coarse-block} gives an integrable $r$th-power tail.  Thus
\eqref{eq:Phi-growth} and dominated convergence show
\[
 \int\Phi(K_{\le J}*f)\longrightarrow\int\Phi(K*f).
\]
Letting $J\to\infty$ in the uniform estimate proves
\eqref{eq:op-phi-estimate}.
\end{proof}

\section{The projected atomic profile}\label{sec:profile}

We now apply Theorem~\ref{thm:operator-phi} to the Riesz/Hodge pair.
The Riesz potential has kernel
\begin{equation}\label{eq:riesz-kernel}
 G_\alpha(x)=c_{n,\alpha}|x|^{\alpha-n},\qquad c_{n,\alpha}\ne0,
\end{equation}
while $P I_\alpha$ has, away from the origin,
\begin{equation}\label{eq:projected-kernel}
 H_\alpha(x)=|x|^{\alpha-n}C_\alpha(x/|x|),
 \qquad C_\alpha\in C^\infty(\Sph^{n-1};\cL(V,V)).
\end{equation}
No distribution supported at $0$ can occur, because such a homogeneous
distribution has degree $-n-k$, whereas $\alpha-n\in(-n,0)$; see
\cite[Chapters~II and V]{Stein70}.  Define
\begin{equation}\label{eq:profile}
 \mathcal A(a)=\int_{\Sph^{n-1}}
 \bigl|C_\alpha(\theta)c_{n,\alpha}^{-1}a\bigr|^r\,d\bar\sigma(\theta).
\end{equation}

The key algebraic fact is the following spherical average.

\begin{lemma}[Spherical mean of the exact projection]\label{lem:average-P}
On $V=\Lambda^\ell\R^n$,
\begin{equation}\label{eq:average-P}
 \int_{\Sph^{n-1}}P(\theta)\,d\bar\sigma(\theta)
 =\frac\ell n\Id_V.
\end{equation}
\end{lemma}

\begin{proof}
We have
\[
 P(\theta)=\sum_{i,j}\theta_i\theta_j\eps_i\ii_j,
 \qquad
 \int_{\Sph^{n-1}}\theta_i\theta_j\,d\bar\sigma=\frac{\delta_{ij}}n.
\]
Hence the average equals $n^{-1}\sum_i\eps_i\ii_i$.  On a basis
$\ell$-blade exactly $\ell$ summands act as the identity, so the sum is
$\ell\Id$.
\end{proof}

For $\ell=1$, this is the familiar identity
$P(\theta)a=\theta\langle\theta,a\rangle$ and
$\int\theta_i\theta_j\,d\bar\sigma=\delta_{ij}/n$.

\begin{lemma}[Coercivity of the profile]\label{lem:profile-coercive}
There is $c=c(n,\ell,r)>0$ such that
\begin{equation}\label{eq:profile-coercive}
 \mathcal A(a)\ge c|a|^r,\qquad a\in V.
\end{equation}
\end{lemma}

\begin{proof}
If $\mathcal A(a)=0$, continuity gives $C_\alpha(\theta)a=0$ for every
$\theta$.  Hence the homogeneous kernel of $P I_\alpha$ applied to $a$
vanishes away from $0$, and therefore identically.  Fourier transform gives
$P(\xi)a=0$ for all $\xi\ne0$.  Averaging and using
Lemma~\ref{lem:average-P} yields $(\ell/n)a=0$.  Thus $\mathcal A$ is positive
on the unit sphere; compactness and $r$-homogeneity give
\eqref{eq:profile-coercive}.
\end{proof}

Define
\begin{equation}\label{eq:Phi-profile}
 \Phi(u,v)=\mathcal A(u)-|v|^r
\end{equation}
on $V\oplus V$, and
\begin{equation}\label{eq:paired-kernel}
 \mathcal K_\alpha(x)a=(G_\alpha(x)a,H_\alpha(x)a).
\end{equation}
Its angular part is
$\widetilde{\mathcal K}_\alpha(\theta)a
=(c_{n,\alpha}a,C_\alpha(\theta)a)$, and the definition of
$\mathcal A$ gives, for every $a\in V$,
\[
 \int_{\Sph^{n-1}}
 \Phi(\widetilde{\mathcal K}_\alpha(\theta)a)\,d\bar\sigma(\theta)=0.
\]
Since $\Phi$ is even, the same holds with the negative kernel.  Applying
Theorem~\ref{thm:operator-phi} to a smooth compactly supported
mean-zero $\omega$ gives
\begin{equation}\label{eq:profile-integral}
 \left|\int_{\R^n}
 \bigl(\mathcal A(I_\alpha\omega)-|P I_\alpha\omega|^r\bigr)\right|
 \lesssim \norm\omega_{L^1}^r.
\end{equation}
By Lemma~\ref{lem:profile-coercive},
\begin{equation}\label{eq:smooth-graph-r}
 \norm{I_\alpha\omega}_{L^r}^r
 \lesssim
 \norm{P I_\alpha\omega}_{L^r}^r+\norm\omega_{L^1}^r.
\end{equation}
This is the desired graph estimate for smooth data.

\section{Closure on the endpoint graph domain}\label{sec:closure}

The passage from \eqref{eq:smooth-graph-r} to arbitrary graph data deserves a
separate argument.  The obstruction is familiar: at the critical exponent
$r=n/(n-\alpha)$ the Riesz potential is only of weak type $(1,r)$ in general,
not strongly bounded $L^1\to L^r$; see \cite[Chapter~V]{Stein70}.  Ordinary
$L^1$ approximation therefore does not preserve the unknown full-potential
term.

\subsection{A quantitative mean-zero estimate}

\begin{lemma}[Mean-zero Schwartz estimate]\label{lem:mean-zero-schwartz}
Let $g\in\cS(\R^n;V)$ satisfy $\int g=0$, and let $N>n+1$.  Then
\begin{equation}\label{eq:mean-zero-schwartz}
 \norm{I_\alpha g}_{L^r}
 \lesssim
 \norm{g}_{L^1}+\norm{g}_{L^\infty}
 +\int_{\R^n}|y|\,|g(y)|\,dy
 +\sup_y(1+|y|)^N|g(y)|.
\end{equation}
\end{lemma}

\begin{proof}
For $|x|\le2$, split the Riesz integral into $|x-y|<1$ and its complement.
The kernel $|x-y|^{\alpha-n}$ is locally integrable, while on the complement
it is bounded.  This part is controlled by
$\norm g_\infty+\norm g_1$.

For $|x|>2$, the zero integral gives
\begin{equation}\label{eq:mean-zero-tail}
 I_\alpha g(x)
 =c_{n,\alpha}\int_{\R^n}
 \bigl(|x-y|^{\alpha-n}-|x|^{\alpha-n}\bigr)g(y)\,dy.
\end{equation}
On $|y|\le|x|/2$, the mean-value theorem bounds the kernel difference by
$C|x|^{\alpha-n-1}|y|$.  On $|y|>|x|/2$, split into
$|x-y|<|x|/2$ and its complement.  In the first region $|y|\simeq|x|$, so the
Schwartz seminorm contributes $C|x|^{\alpha-N}$ after integration.  In the
second region both kernels are $O(|x|^{\alpha-n})$, and
\[
 \int_{|y|>|x|/2}|g(y)|\,dy
 \le \frac2{|x|}\int|y|\,|g(y)|\,dy.
\]
Thus the tail is bounded by a constant times
$|x|^{\alpha-n-1}+|x|^{\alpha-N}$ multiplied by the quantities on the right of
\eqref{eq:mean-zero-schwartz}.  Since
$(n+1-\alpha)r=n+r>n$ and $N>n+1$, both powers belong to
$L^r(\{|x|>2\})$.
\end{proof}

\subsection{Annular graph density}
Let $q\in C_c^\infty(\R^n\setminus\{0\})$ and let $q(D)$ be the corresponding
Fourier multiplier.

\begin{lemma}[Annular graph density]\label{lem:annular-density}
For every $f\in L^1(\R^n;V)$, the form $q(D)f$ lies in the closure of
\[
 \{h\in C_c^\infty(\R^n;V):\int h=0\}
\]
for the norm
\begin{equation}\label{eq:graph-density-norm}
 \norm h_{L^1}+\norm{I_\alpha h}_{L^r}
 +\norm{P I_\alpha h}_{L^r}.
\end{equation}
\end{lemma}

\begin{proof}
Choose $f_k\in C_c^\infty(\R^n;V)$ with $f_k\to f$ in $L^1$.  The multipliers
$q(D)$ and $I_\alpha q(D)$ have Schwartz kernels because their symbols are
smooth and compactly supported away from the origin.  Young's inequality gives
\[
 \norm{q(D)(f_k-f)}_{L^1}\to0,
 \qquad
 \norm{I_\alpha q(D)(f_k-f)}_{L^r}\to0.
\]
The zero-order multiplier $P$ is bounded on $L^r$, so the projected term also
converges.

It remains to approximate $h=q(D)f_k$.  Its Fourier transform vanishes near
$0$, hence $h\in\cS_0$ and $\int h=0$.  Let $\eta\in C_c^\infty$ equal $1$ on
the unit ball, set $\eta_R(x)=\eta(x/R)$, and choose
$\psi\in C_c^\infty$ with $\int\psi=1$.  Define
\begin{equation}\label{eq:cutoff-correction}
 h_R=\eta_R h-\left(\int\eta_Rh\right)\psi.
\end{equation}
Then $h_R$ is smooth, compactly supported, and mean zero.  Since
$\int h=0$, the correction coefficient is
$-\int(1-\eta_R)h$ and decays faster than any negative power of $R$.
Consequently $h_R\to h$ in $\cS$ and in $L^1$.  Applying
Lemma~\ref{lem:mean-zero-schwartz} to $h-h_R$ gives
$I_\alpha(h-h_R)\to0$ in $L^r$, and $L^r$ boundedness of $P$ gives the
projected convergence.
\end{proof}

\subsection{Removing the annular cutoff}
Choose a radial $\chi\in C_c^\infty(\R^n)$ with $\chi=1$ near $0$ and, for
$N\ge1$, set
\begin{equation}\label{eq:QN}
 Q_N=\chi(2^{-N}D)-\chi(2^N D).
\end{equation}
The symbol of $Q_N$ is smooth, compactly supported, and vanishes near the
origin.  The operators $Q_N$ commute with $P$ and $I_\alpha$, are uniformly
bounded on $L^1$ and $L^r$, converge strongly to the identity on $L^r$, and
converge to the identity on homogeneous distributions when tested against
$\cS_0$.  Since the symbol of $Q_N$ vanishes near the origin, its convolution
kernel has integral zero; hence $Q_N\omega\in L^1_0$ whenever
$\omega\in L^1$.

Let $\omega\in L^1_0(\R^n;V)$ and assume
$P I_\alpha\omega\in L^r$.  For each fixed $N$, apply
Lemma~\ref{lem:annular-density} to the annular multiplier $Q_N$ and the input
$\omega$.  The smooth graph estimate \eqref{eq:smooth-graph-r} therefore
passes to $Q_N\omega$ in the graph norm.  Hence
\begin{align}\label{eq:QN-bound}
 \norm{Q_N I_\alpha\omega}_{L^r}^r
 &\lesssim
 \norm{Q_N P I_\alpha\omega}_{L^r}^r
 +\norm{Q_N\omega}_{L^1}^r\notag\\
 &\lesssim
 \norm{P I_\alpha\omega}_{L^r}^r+\norm{\omega}_{L^1}^r,
\end{align}
with a constant independent of $N$.

Since $1<r<\infty$, $L^r$ is reflexive.  Take a subsequence realizing the
lower limit of the left side and then a weakly convergent further subsequence.
Testing against $\cS_0$, distributional convergence of
$Q_N I_\alpha\omega$ identifies the weak limit with the same homogeneous
distribution class as $I_\alpha\omega$.  The weak limit belongs to $L^r$;
any polynomial ambiguity is therefore zero.  Weak lower semicontinuity in
\eqref{eq:QN-bound} gives
\begin{equation}\label{eq:full-graph-r}
 \norm{I_\alpha\omega}_{L^r}^r
 \lesssim
 \norm{P I_\alpha\omega}_{L^r}^r+\norm{\omega}_{L^1}^r.
\end{equation}
Taking $r$th roots proves \eqref{eq:graph}.

\section{Duality and bounded representatives}\label{sec:duality}

We now convert \eqref{eq:graph} into the selection theorem.  The only
low-frequency fact needed is that the intersection $X\cap L^1$ is
automatically mean zero.

\begin{lemma}\label{lem:critical-mean-zero}
If $f\in X\cap L^1(\R^n)$, then $\int_{\R^n}f=0$.
\end{lemma}

\begin{proof}
Choose nonzero $\kappa\in\cS_0$ and put
$\Delta_jf=\kappa_j*f$, $\kappa_j(x)=2^{jn}\kappa(2^jx)$.
Writing $f=|D|^\alpha g$ with $g=I_\alpha f\in L^r$, annular multiplier
scaling gives
\begin{equation}\label{eq:low-frequency-vanish}
 2^{-j\alpha}\norm{\Delta_jf}_{L^r}\longrightarrow0
 \qquad(j\to-\infty).
\end{equation}
On the other hand, since $\alpha=n(1-1/r)$,
\begin{equation}\label{eq:low-frequency-mass}
 2^{-j\alpha}\norm{\Delta_jf}_{L^r}
 =\left\|\int_{\R^n}\kappa(\,\cdot-2^jy\,)f(y)\,dy\right\|_{L^r}.
\end{equation}
Translation continuity in $L^r$ and dominated convergence show that the
right side tends to $\norm\kappa_{L^r}|\int f|$.  Comparison with
\eqref{eq:low-frequency-vanish} proves the claim.
\end{proof}

Hence $X$ and $L^1_0$ are compatible homogeneous distribution spaces.  Their
sum has dual
\begin{equation}\label{eq:sum-dual}
 (X+L^1_0)'=Y\cap(L^\infty/\R)
\end{equation}
with equivalent norms: a functional on the sum restricts to elements of
$X'=Y$ and $(L^1_0)'=L^\infty/\R$, and agreement on $\cS_0$ identifies the
two representatives; the converse follows from the infimum norm.

\begin{proposition}[Graph estimate implies bounded selection]\label{prop:graph-to-selection}
Assume \eqref{eq:graph}.  Then every $v\in Y\Lambda^\ell$ admits
$u\in(Y\cap L^\infty)\Lambda^\ell$ such that
\[
 du=dv,\qquad \norm u_{Y\cap L^\infty}\lesssim\norm v_Y.
\]
\end{proposition}

\begin{proof}
We first extend the graph estimate to the coexact part of
$X+L^1_0$.  Let $\eta\in(X+L^1_0)\Lambda^\ell$ satisfy
$P^\perp\eta=\eta$, and write $\eta=\eta_0+\eta_1$ with
$\eta_0\in X$, $\eta_1\in L^1_0$.  Since $P\eta=0$,
$P\eta_1=-P\eta_0\in X$.  Applying \eqref{eq:graph}, in the equivalent
$X$-norm \eqref{eq:X-Riesz}, to $\eta_1$ gives
\[
 \norm{\eta_1}_X
 \lesssim \norm{\eta_1}_{L^1}+\norm{P\eta_0}_X
 \lesssim \norm{\eta_1}_{L^1}+\norm{\eta_0}_X.
\]
Thus $\eta\in X$ and, after taking the infimum over all decompositions,
\begin{equation}\label{eq:coexact-sum-estimate}
 \norm\eta_X\lesssim\norm\eta_{X+L^1_0}.
\end{equation}

Let
$\mathcal E=\{\eta\in X\Lambda^\ell:P^\perp\eta=\eta\}$.
For fixed $v\in Y\Lambda^\ell$, the functional
$L_v(\eta)=\pair{\eta}{v}$ on $\mathcal E$ satisfies, by
\eqref{eq:coexact-sum-estimate},
\[
 |L_v(\eta)|
 \le\norm\eta_X\norm v_Y
 \lesssim\norm\eta_{X+L^1_0}\norm v_Y.
\]
Hahn--Banach extends it to $(X+L^1_0)\Lambda^\ell$ with the same bound.
By \eqref{eq:sum-dual}, the extension is represented by some
$u\in(Y\cap L^\infty)\Lambda^\ell$ with
$\norm u_{Y\cap L^\infty}\lesssim\norm v_Y$.

For every $\varphi\in\cS_0\Lambda^\ell$, $P^\perp\varphi\in\mathcal E$, so
\[
 \pair{\varphi}{P^\perp(v-u)}
 =\pair{P^\perp\varphi}{v-u}=0.
\]
Hence $P^\perp(v-u)=0$.  From \eqref{eq:cartan},
\[
 |\eps_\xi w|^2
 =|\xi|^2\pair{P^\perp(\xi)w}{w},
\]
so $P^\perp(\xi)w=0$ implies $\eps_\xi w=0$.  Therefore
$d(v-u)=0$, i.e. $du=dv$.
\end{proof}

\begin{proof}[Proof of Theorem~\ref{thm:main}]
The full graph estimate \eqref{eq:full-graph-r} proves \eqref{eq:graph};
Proposition~\ref{prop:graph-to-selection} gives a homogeneous bounded class
$u$ with $du=dv$ and
\[
 \norm u_Y+\norm u_{L^\infty/\R}\lesssim\norm v_Y.
\]
Choose a constant representative as after
\eqref{eq:Linfty-mod-constants}.  This preserves $du=dv$ and gives
\eqref{eq:main-estimate}.  Equality of the image spaces and equivalence of
their quotient norms follow immediately.
\end{proof}

\section{A consequence}\label{sec:consequences}

\begin{corollary}[Bounded-potential Hodge decomposition]\label{cor:bounded-hodge}
Let $n\ge4$, $2\le\ell\le n-2$, and $1<p<\infty$.  Then
\begin{equation}\label{eq:bounded-hodge}
\begin{aligned}
 \dot W^{n/p-1,p}(\R^n;\Lambda^\ell)
 &={}d\bigl[(\dot W^{n/p,p}\cap L^\infty)(\R^n;\Lambda^{\ell-1})\bigr]\\
 &\quad\oplus
 d^*\bigl[(\dot W^{n/p,p}\cap L^\infty)(\R^n;\Lambda^{\ell+1})\bigr]
\end{aligned}
\end{equation}
with equivalent induced norms.
\end{corollary}

\begin{proof}
The homogeneous Hodge decomposition gives
\[
 \dot W^{n/p-1,p}\Lambda^\ell
 =d[\dot W^{n/p,p}\Lambda^{\ell-1}]
 \oplus d^*[\dot W^{n/p,p}\Lambda^{\ell+1}].
\]
Apply Theorem~\ref{thm:main} to the exact term and, using
$d^*=\pm\star d\star$, to the coexact term.  The degree restrictions ensure
that both applications are in intermediate form degree, and the quotient
bounds give the norm equivalence.
\end{proof}

\section*{Statement and declaration}
The derivations in this work are assisted by the TARS system via exploratory reasoning. X. Dai supplemented critical argument details, refined the manuscript logic, and completed the writing.

\raggedbottom

\end{document}